\documentclass[reqno,10pt,a4paper]{article}

\usepackage{a4wide}
\usepackage{enumitem}
\setlist{itemsep=2pt,parsep=0pt,topsep=2pt,labelsep=.5em,leftmargin=\parindent}

\usepackage{titlesec}

\titleformat{\section}[block]
{\normalfont\Large\filcenter\bfseries}{\thesection.}{.33em}{}
\titlespacing*{\section}
{0pt}{3.5ex plus 1ex minus .2ex}{2.3ex plus .2ex}

\titleformat{\subsection}[runin]
{\normalfont\normalsize\bfseries}{\thesubsection.}{.33em}{}[.]
\titlespacing*{\subsection}
{0pt}{3.25ex plus 1ex minus .2ex}{.5em}

\usepackage{amsmath}
\usepackage{amsthm}
\usepackage{amsfonts}
\usepackage{amssymb}
\usepackage[all]{xy}
\SelectTips{cm}{10}
\usepackage{url}

\theoremstyle{plain}
\newtheorem{theorem}{Theorem}
\newtheorem{lemma}[theorem]{Lemma}

\newtheorem{proposition}[theorem]{Proposition}

\newtheorem*{scorollary}{Corollary}

\theoremstyle{definition}

\newtheorem{example}[theorem]{Example}

\theoremstyle{remark}

\newcommand{\mcC}{\mathcal{C}}
\newcommand{\mcD}{\mathcal{D}}
\newcommand{\mcI}{\mathcal{I}}
\newcommand{\mcM}{\mathcal{M}}
\newcommand{\mcN}{\mathcal{N}}
\newcommand{\mcP}{\mathcal{P}}
\newcommand{\mcQ}{\mathcal{Q}}
\newcommand{\mcR}{\mathcal{R}}
\newcommand{\mcV}{\mathcal{V}}

\newcommand{\bbZ}{\mathbb{Z}}

\newcommand{\sfB}{\mathsf{B}}
\newcommand{\sfC}{\mathsf{C}}
\newcommand{\sfW}{\mathsf{W}}

\newcommand{\cosimp}{\mathsf{c}}
\newcommand{\simp}{\mathsf{s}}

\newcommand{\Ch}{\mathsf{Ch}}
\newcommand{\Set}{\mathsf{Set}}
\newcommand{\sSet}{\simp\Set}

\newcommand{\bbone}{\mathsf I}

\newcommand{\colim}{\operatorname{colim}}
\newcommand{\const}{\operatorname{cst}}
\newcommand{\Ho}{\operatorname{Ho}}
\newcommand{\id}{\operatorname{id}}
\newcommand{\Id}{\operatorname{Id}}
\newcommand{\pcm}{\operatorname{pcm}}
\newcommand{\sk}{\operatorname{sk}}
\newcommand{\ra}{\rightarrow}
\newcommand{\simarrow}{\mathbin{\:\!\!\xymatrix@1@C=15pt{{}\ar[r]^{\sim} & {}}}}
\renewcommand{\rightarrowtail}[1][]{\mathbin{\:\!\!\xymatrix@1@C=15pt{{}\ar@{ >->}[r]^{#1} & {}}}}

\renewcommand{\deg}{\mathrm{deg}}
\newcommand{\op}{\mathrm{op}}
\newcommand{\nd}{\mathrm{nd}}
\newcommand{\nmono}{\mathrm{n\pm}}
\newcommand{\nneg}{\mathrm{n-}}
\newcommand{\npos}{\mathrm{n+}}
\newcommand{\pt}{\mathrm{pt}}
\newcommand{\Reedy}{\mathrm{R}}

\newcommand{\col}{\!:\:\!}
\newcommand{\loc}{\!\::\!}

\newcommand{\lra}{\longrightarrow}
\newcommand{\xra}[1]{\xrightarrow{#1}}
\newcommand{\xla}[1]{\xleftarrow{#1}}
\newcommand{\xlra}[1]{\xrightarrow{\ #1\ }}
\newcommand{\Ra}{\Rightarrow}
\newcommand{\hra}{\hookrightarrow}

\newdir{c}{{}*!/-5pt/@^{(}}
\newdir{d}{{}*!/-5pt/@_{(}}
\newdir{ >}{{}*!/-5pt/@{>}}
\newdir{s}{{}*!/+10pt/@{}}
\newdir{|>}{{}*!/2pt/@{|}*@{>}}

\newcommand{\xymatrixc}[2][]{\xy *!C\xybox{\xymatrix#1{#2}}\endxy}
\newcommand{\xymatrixb}[2]{\xy*i\xybox{\xymatrix{#2}};p!D*!D\xybox{\xymatrix{#1}}\endxy}

\newcommand{\topbox}[3]{\xybox{*!{\hphantom{#1}\phantom{#2}\hphantom{#3}},c*h!{#1#2#3}} \save *{\vphantom{{#1}{#2}{#3}}} \restore}
\newcommand{\rbox}[2]{\drop+<.5pc>!!<0pt,\the\fontdimen22\textfont2>{\vphantom{#1#2}}\drop\xybox{*i+<.5pc>!!<0pt,\the\fontdimen22\textfont2>{#2},c+R*+<.5pc>!!R!<0pt,\the\fontdimen22\textfont2>{\hphantom{#1}#2},c*h+<.5pc>!!R!<0pt,\the\fontdimen22\textfont2>{#1#2}}}
\newcommand{\cbox}[1]{\drop+<.5pc>!!<0pt,\the\fontdimen22\textfont2>{#1}}

\newcommand{\lcrbox}[3]{\drop+<.5pc>!!<0pt,\the\fontdimen22\textfont2>{\vphantom{#1#2#3}}\drop\xybox{*+<.5pc>!!<0pt,\the\fontdimen22\textfont2>{#2};p+R*+<.5pc>!!R!<0pt,\the\fontdimen22\textfont2>{\hphantom{#1#2}},
c*h+<.5pc>!!R!<0pt,\the\fontdimen22\textfont2>{#1\phantom{#2}},
p+L*+<.5pc>!!L!<0pt,\the\fontdimen22\textfont2>{\hphantom{#2#3}},
c*h+<.5pc>!!L!<0pt,\the\fontdimen22\textfont2>{\phantom{#2}#3}}}

\newcommand{\leftbox}[2]{{}\phantom{#1} \save []+L*+<.5pc>!!<0pt,\the\fontdimen22\textfont2>!L{#1#2} \restore}
\newcommand{\rightbox}[2]{{}\phantom{#2} \save []+R*+<.5pc>!!<0pt,\the\fontdimen22\textfont2>!R{#1#2} \restore}

\newcommand{\pbsize}{15pt}
\newcommand{\pboff}{.5}
\newcommand{\poxy}{\save +<-\pbsize,0pt>;p+<0pt,\pbsize>="a"**{}?(\pboff);"a"**\dir{-};p+<\pbsize,0pt>;**{}?(\pboff);"a"**\dir{-} \restore}
\newcommand{\po}{\poxy}

\theoremstyle{plain}
\newtheorem{atheorem}{Theorem}

\begin{document}

\title{Enriched cofibration categories\thanks{%
The research was supported by the grant P201/11/0528 of the Czech Science Foundation (GA \v CR).\newline
\emph{2010 Mathematics Subject Classification}. Primary 55U35; Secondary 18D20.\newline
\emph{Key words and phrases}. cofibration category, diagram category, weighted colimit.
}}
\author{Luk\'a\v{s} Vok\v{r}\'inek}
\date{\today}
\maketitle

\begin{abstract}
Cofibration categories are a formalization of homotopy theory useful for dealing with homotopy colimits that exist on the level of models as colimits of cofibrant diagrams. In this paper, we deal with their enriched version. Our main result claims that the category $[\mcC,\mcM]$ of enriched diagrams equipped with the projective structure inherits a structure of a cofibration category whenever $\mcC$ is locally cofibrant (or, more generally, locally flat).
\end{abstract}

\section{Introduction}

A cofibration category $\mcM$ is a category equipped with a collection of cofibrations and weak equivalences in such a way that it is possible to construct all homotopy colimits on the level of $\mcM$ (rather than $\Ho\mcM$). It is known that the category $[\mcC,\mcM]$ of diagrams possesses the so-called pointwise cofibration structure -- both cofibrations and weak equivalences are pointwise. In this paper, we improve on this result in two different ways:
\begin{itemize}
\item
	we endow the diagram category with a stronger projective structure and show that the colimit functor is left Quillen;
\item
	we consider both $\mcC$ and $\mcM$ enriched over a monoidal category $\mcV$ and $[\mcC,\mcM]$ now denotes the category of enriched diagrams; the same is then true provided that $\mcC$ is ``locally flat'' (in classical enriched homotopy theory, locally cofibrant would be sufficient).
\end{itemize}

We provide an example with $\mcV = \Ch_\bbZ$, the category of chain complexes, and $\mcC$ non-locally flat for which $[\mcC,\mcM]$ is \emph{not} a cofibration category. The only problem lies in the cofibrant replacement in this category. If a functorial (in the enriched sense) cofibrant replacement existed in $\Ch_\bbZ$, then one could produce a cofibrant replacement in $[\mcC,\mcM]$ for arbitrary $\mcC$. Thus, our example could be reinterpreted as the non-existence of a functorial cofibrant replacement in $\Ch_\bbZ$. This is in strong contrast with the nowadays frequent assumption of non-enriched model categories possessing a functorial cofibrant replacement (factorization).

Namely, in the category $\Ch_\bbZ$ of chain complexes over $\bbZ$ with the projective model structure, consider the object $\bbZ/2$ which clearly satisfies $2\cdot\id_{\bbZ/2}=0$. If there should be a cofibrant replacement dg-functor, it would have to be additive and, in particular, if $P$ is its value on $\bbZ/2$, then necessarily $2\cdot\id_P=0$. Since every cofibrant object is also torsion-free, we conclude that $P=0$ and thus, $P$ might not be a replacement of $\bbZ/2$. We will see shortly that this happens precisely because $\Ch_\bbZ(\bbZ/2,\bbZ/2)$ is not cofibrant.

There are many choices of cofibrations in $[\mcC,\mcM]$. The most obvious is the class of pointwise cofibrations, i.e.\ natural transformations $f \colon X \to Y$ such that each $f_c \colon Xc\ra Yc$ is a cofibration. Together with pointwise weak equivalences, these form the pointwise structure on $[\mcC,\mcM]$ that we denote $[\mcC,\mcM]_\pt$. However, this class is too big for functors like the colimit to preserve cofibrations. In the presense of fibrations, another possibility would be to consider all maps that have a left lifting property with respect to trivial fibrations. Again, it is very hard to verify whether the colimit functor preserves cofibrations since we will not assume in general that cofibrations in $\mcM$ are characterized by the lifting property.

Since we are interested mainly in colimits/left Kan extensions, the above choices are inappropriate. A \emph{(projective) cofibration} $f \colon X\ra Y$ in $[\mcC,\mcM]$ is a map that lies in the closure of
\[\{\mcC(c,-)\otimes M\ra\mcC(c,-)\otimes N\ |\ \textrm{$c\in\mcC$ and $M\ra N$ a cofibration in $\mcM$}\}\]
under coproducts, cobase change and transfinite composites. When we speak about cofibrations (without any adjective) in diagram categories, we will always mean projective cofibrations.

We are now almost ready to state our main theorem. We leave certain notions undefined for now, but remark that an important example is that of a monoidal model category $\mcV$, a model $\mcV$-category $\mcM$ and a locally cofibrant $\mcV$-category $\mcC$.

\begin{atheorem}\label{t:projective_structure}
Given a cofibration $\mcV$-category $\mcM$ and a small locally flat $\mcV$-category $\mcC$, the diagram category $[\mcC,\mcM]$ endowed with projective cofibrations and pointwise weak equivalences becomes a cofibration $\mcV$-category.
\end{atheorem}

To some extent, it seems a bit unfortunate that we do not have a description of trivial projective cofibrations via a generating collection as for projective cofibrations. However, in the general context of cofibration categories, when trivial cofibrations are not determined by lifting properties, it might possibly happen that such a description is not true. On the other hand, we were able to provide a description of weak equivalences between cofibrant objects.

\begin{atheorem}\label{t:weak_equivalences_generation}
Let $\mcC$ be a small locally flat $\mcV$-category. Then the class of weak equivalences between cofibrant objects in $[\mcC,\mcM]$ is the smallest class of morphisms between cofibrant objects which
\begin{itemize}
\item contains all morphisms $\mcC(c,-)\otimes M\ra\mcC(c,-)\otimes N$ with $M\ra N$ a weak equivalence between cofibrant objects in $\mcM$,
\item satisfies the 2-out-of-3 property and
\item is closed under ``homotopy invariant'' colimits:
\begin{itemize}
\item if $f \colon X \to Y$ is a natural transformation between spans $X_1 \xla{x_1} X_0 \xra{x_2} X_2$, $Y_1 \xla{y_1} Y_0 \xra{y_2} Y_2$ of cofibrant objects with one of the maps $x_1$, $x_2$ and one of the maps $y_1$, $y_2$ a cofibration, and if the components of $f$ lie in the class then so does the map of colimits.
\item if $f \colon X \to Y$ is a natural transformation between cofibrant chains (all objects cofibrant, all maps cofibrations), whose components lie in the class then so does the map of colimits.
\end{itemize}
\end{itemize}
\end{atheorem}

This implies easily that the (weighted) colimit functor is left Quillen; of course, such a result would not be true for the pointwise structure.

\begin{scorollary}
The weighted colimit functor $W\otimes_\mcC- \colon [\mcC,\mcM] \to \mcM$ is left Quillen for any pointwise flat weight $W$.
\end{scorollary}

\begin{proof}
The weighted colimit satisfies $W \otimes_\mcC (\mcC(c,-) \otimes M) \cong Wc \otimes M$ and thus takes the generating weak equivalences to weak equivalences by Brown's lemma. It also commutes with the constructions from Theorem~\ref{t:weak_equivalences_generation}.
\end{proof}

\section{A very small example}

We believe that it is instructive to produce concretely cofibrant replacements of some very small diagrams before digging into the general proof. Here we present the main idea in the case that there are no endomorphisms present in $\mcC$. In order to avoid defining cofibration categories at this point, we will assume that $\mcM$ is a model $\mcV$-category.

Let there be given a diagram $X \colon \mcC\ra\mcM$ as above with $\mcC$ locally cofibrant. We will try to show on a small portion of $\mcC$ what needs to be done in order to find a cofibrant replacement of $X$. Let there be given two objects $c_0$ and $c_1$ of $\mcC$ and assume that $\mcC(c_1,c_0)=0$ and $\mcC(c_0,c_0)=\bbone=\mcC(c_1,c_1)$. Then, on this portion of $\mcC$, the diagram $X$ is given by two objects $Xc_0$ and $Xc_1$ together with a map $\mcC(c_0,c_1)\otimes Xc_0\ra Xc_1$. Choose any cofibrant replacements $\widetilde Xc_0$ and $\widetilde Xc_1$ of them. Then we have a diagram
\[\xymatrix{
\widetilde Xc_0 \ar[d]_{\varepsilon_0}^\sim & & \widetilde Xc_1 \ar[d]^{\varepsilon_1}_\sim \\
Xc_0 \ar@/_/[rr] \ar@/^/[rr] & {\scriptstyle \mcC(c_0,c_1)} & Xc_1 \\
}\]
where the double arrow with $\mcC(c_0,c_1)$ in between indicates that there is a family of morphisms from $Xc_0$ to $Xc_1$ parametrized by this object. By composing with $\varepsilon_0$ one obtains a similar family of morphisms from $\widetilde Xc_0$ to $Xc_1$. Viewing this family as a single map $\mcC(c_0,c_1)\otimes\widetilde Xc_0\ra Xc_1$ we recall our cofibrancy conditions and observe that the domain is still cofibrant. If the map $\varepsilon_1$ was actually a trivial fibration we would obtain the desired family of maps $\widetilde Xc_0\ra\widetilde Xc_1$ using the lifting axiom. This is sufficient for this simple example but not very useful when compositions of morhpisms appear. Let us therefore assume that $\varepsilon_1$ is only a weak equivalence. Let us factor
\[(\mcC(c_0,c_1)\otimes\widetilde Xc_0)+\widetilde Xc_1\lra Xc_1\]
into a cofibration followed by a weak equivalence $\widetilde X(c_0,c_1)\xra{\sim}Xc_1$. We obtain a diagram
\[\xymatrix{
\widetilde Xc_0 \ar[d]_{\varepsilon_0}^\sim \ar@/_/[rr] \ar@/^/[rr] & {\scriptstyle \mcC(c_0,c_1)} & \widetilde X(c_0,c_1) \ar[d]_\sim & \widetilde Xc_1 \ar[dl]^\sim \ar@{ >->}[l]_-\sim \\
Xc_0 \ar@/_/[rr] \ar@/^/[rr] & {\scriptstyle \mcC(c_0,c_1)} & Xc_1 \\
}\]
In this small example we obtain a cofibrant replacement given by $\widetilde Xc_0$ and $\widetilde X(c_0,c_1)$.

If we have three objects $c_0$, $c_1$, $c_2$ and non-identity morphisms only going ``up'' then doing the above separately for pairs $c_0$, $c_2$ and $c_1$, $c_2$ we obtain two candidates for a cofibrant replacement of $Xc_2$, namely, $\widetilde X(c_0,c_2)$ and $\widetilde X(c_1,c_2)$. They are a part of the following diagram of cofibrant objects over $Xc_2$.
\[\xymatrix@!=30pt{
& & \widetilde Xc_2 \ar@{ >->}[dl]_-\sim \ar@{ >->}[dr]^-\sim \\
& \widetilde X(c_0,c_2) & & \widetilde X(c_1,c_2) \\
\rightbox{\mcC(c_0,c_1,c_2)}{{}\otimes\widetilde Xc_0} \ar[ur] \ar@{ >->}[rr] & & \mcC(c_1,c_2)\otimes\widetilde X(c_0,c_1) & & \leftbox{\mcC(c_1,c_2)}{{}\otimes\widetilde Xc_1} \ar@{ >->}[ul] \ar@{ >->}[ll]
}\]
where $\mcC(c_0,c_1,c_2)$ denotes $\mcC(c_1,c_2)\otimes\mcC(c_0,c_1)$. Taking a factorization of the canonical map from the colimit of this diagram (which happens to be cofibrant!) to $Xc_2$ results in a cofibration followed by a weak equivalence $\widetilde X(c_0,c_1,c_2)\ra Xc_2$. It should perhaps be clear how to continue, at least when no endomorphisms appear. We will proceed with details in the general case in Section~\ref{s:proof_main}.

\section{Conventions and notations}

The category $\Delta$ consists of finite non-empty ordinals that we denote $[n] = \{0 < \cdots <n\}$. In particular, $[1] = \{0 < 1\}$ is the category with one arrow and thus, $\mcM^{[1]}$ denotes the category of maps in $\mcM$.

We will denote the representable functors as follows
\[\mcC_c = \mcC(c,-) \in [\mcC,\mcV], \quad \mcC^c = \mcC(-,c) \in [\mcC^\op,\mcV].\]
In particular, the standard simplex will be denoted by $\Delta^n$.

We will reserve $\mcV$ for a symmetric monoidal category, $\mcM$ will be a cofibration category/model category, later also equipped with an action of $\mcV$. The shape of the diagrams will be denoted either by $\mcD$ (direct category), $\mcR$ (Reedy category) or $\mcC$ (arbitrary enriched category).

We refer the reader to \cite{Kelly} for basics about enriched category theory, the definition and properties of a coend, as well as its special cases -- the weighted colimit and the left Kan extension. We denote the (enriched) coend $\smallint^\mcC \colon [\mcC^\op \otimes \mcC, \mcM] \to \mcM$; the coend $\smallint^\mcC F$ of a bifunctor $F$ is given as a coequalizer
\[\xymatrix{
\sum_{c_0,c_1 \in \mcC} \mcC(c_0,c_1) \otimes F(c_1,c_0) \ar@<-2pt>[r] \ar@<2pt>[r] & \sum_{c \in \mcC} F(c,c) \ar[r] & \smallint^\mcC F
}\]
(the maps apply the morphism from $\mcC(c_0,c_1)$ to one of the arguments of $F$, thus mapping to the summand $\mcC(c_i,c_i)$).

When there is given a bifunctor $\odot \colon \mcP \otimes \mcQ \to \mcM$, there is an induced bifunctor
\[\odot_\mcC \colon [\mcC^\op,\mcP] \otimes [\mcC,\mcQ] \to \mcM,\]
given as $W \odot_\mcC D = \smallint^\mcC W \odot D$, the coend of $\mcC^\op \otimes \mcC \xra{W \otimes D} \mcP \otimes \mcQ \xra{\odot} \mcM$. In the special case of an action $\otimes \colon \mcV \otimes \mcM \to \mcM$, the so obtained $W \otimes_\mcC D$ is the \emph{weighted colimit} of $D$ (weighted by $W$).

\section{Categories with cofibrations}

All categories of homotopical nature will be assumed to contain an initial object which we denote by $0$ and all functors between such categories will be assumed to preserve the initial objects.

Before proving the main theorems, we will present some preparatory material. We start with results that only concern cofibrations. These are summarized in this section.

\subsection*{Left closed classes}

Let $\mcM$ be a category. A subcategory $\sfC\subseteq\mcM$ is said to be \emph{left closed class} if it satisfies the following set of axioms in which we call morphisms of $\sfC$ \emph{cofibrations} and objects of $\sfC$ \emph{good}
\begin{itemize}
\item Coproducts of good objects exist in $\mcM$ and are good; in particular, $0$ is good.
\item All isomorphisms with good domains are cofibrations (hence, their codomains are also good).
\item A pushout of a cofibration along a map between good objects exists in $\mcM$ and is again a cofibration.
\item A transfinite composite of cofibrations exists in $\mcM$ and is again a cofibration.
\end{itemize}
An object $M$ is called \emph{cofibrant} if the unique map $0\ra M$ is a cofibration. In particular, each cofibrant object is good. A category equipped with a left closed class will be called a \emph{category with cofibrations}.

Upon replacing $\sfC$ by its full subcategory on cofibrant objects, we obtain another category with cofibrations which is \emph{reduced}: all good objects are cofibrant. Its cofibrations, i.e.\ cofibrations between cofibrant objects in $\mcM$, are called \emph{strong cofibrations} of $\mcM$. We will usually assume $\mcM$ to be reduced. However, the non-reduced case is important since it includes the case of trivial cofibrations in cofibration categories. Another example, where this level of generality is useful, are the model categories -- there all objects might be assumed to be good. A further reason for not disregarding non-cofibrant objects is that in a monoidal model category, one does not usually assume the unit of the monoidal structure to be cofibrant and much of the theory could be reproduced if it is merely good.

It is possible to \emph{generate} cofibrations by a collection of morphisms (by closing it under coproducts, pushouts and transfinite composites), but one must specify good objects in advance to specify the allowable pushouts. From this perspective, it is much simpler to generate strong cofibrations. The collection of \emph{(projective) cofibrations} of $[\mcC,\mcM]$ is generated by the $\mcC_c \otimes M \to \mcC_c \otimes N$ with $M \to N$ a cofibration in $\mcM$, where the good objects are the pointwise good diagrams.

\begin{proposition}\label{p:generating_strong_cofibrations}
Strong cofibrations in $[\mcC,\mcM]$ are generated by
\[\{\mcC_c\otimes M\ra\mcC_c\otimes N\ |\ \textrm{$c\in\mcC$ and $M\ra N$ a strong cofibration in $\mcM$}\}.\]
\end{proposition}

\begin{proof}
This follows from the fact that in the chain $Y_n$ presenting a strong cofibration, all objects are cofibrant and therefore pointwise cofibrant. We have the following factorization
\[\xymatrix{
\mcC_c \otimes M \ar[r] \ar[d] & \mcC_c \otimes Y_{n-1}c \ar[r] \ar[d] & Y_{n-1} \ar[d] \\
\mcC_c \otimes N \ar[r] & \mcC_c \otimes N' \ar[r] \po & Y_n \po
}\]
where $N'$ is the pushout $Y_{n-1}c +_M N$ and thus, $Y_{n-1}c \to N'$ is a strong cofibration, as required.
\end{proof}

\subsection*{Left closed functors}

A functor $F \colon \mcM \ra \mcN$ between categories with cofibrations is called \emph{left closed} if it preserves strong cofibrations. It is easy to see that, if $F$ preserves suitable colimits, it is enough to verify this condition for a generating class of strong cofibrations.

The axioms of a left closed class allow one to define cofibrations in the category $\mcM^{[1]}$ of arrows in $\mcM$. A morphism from $f$ to $g$ is a cofibration if and only if in the corresponding square
\[\xymatrix{
A \ar[r]^i \ar[d]_f & B \ar[d]^g \\
X \ar[r]_j & Y
}\]
all of $i$, $j$ and the puhsout corner map are cofibrations; the pushout corner map is the map from the pushout to the terminal object, see Section~\ref{s:cubical_diagrams}.
An object $f\in\mcM^{[1]}$ is cofibrant if and only if it is a strong cofibration. A strong cofibration in $\mcM^{[1]}$ is a \emph{cofibrant square}, i.e.\ a square in which all objects are cofibrant, all maps are cofibrations and so is the pushout corner map.

Thus, a functor $\mcM\ra\mcN$ is left closed if and only if $\mcM^{[1]}\ra\mcN^{[1]}$ preserves cofibrant objects. There is a similar structure on $\mcM^{[1]\times[1]}$ whose cofibrant objects are precisely the cofibrant squares.

A bifunctor $\odot \colon \mcP \otimes \mcQ \ra \mcN$ is said to be a \emph{left closed bifunctor} if, for a strong cofibration $V \ra W$ in $\mcP$ and a strong cofibration $X\ra Y$ in $\mcQ$, the square
\[\xymatrix{
V\odot X \ar[r] \ar[d] & W\odot X \ar[d] \\
V\odot Y \ar[r] & W\odot Y
}\]
is cofibrant, i.e.\ if the induced bifunctor $\mcP^{[1]} \otimes \mcQ^{[1]} \ra \mcN^{[1]\times[1]}$ preserves cofibrant objects. Again, if the original bifunctor $\odot$ preserves suitable colimits in both variables, it is enough to verify the condition only for a generating class of strong cofibrations in each argument.

Put differently, $\odot$ is a left closed bifunctor if and only if for each strong cofibration $i \colon V\ra W$ of $\mcP$, the induced functor $i\odot- \colon \mcQ\ra\mcN^{[1]}$ is left closed. We will now revert this and produce a left closed class from such a bifunctor.

\subsection*{Auxiliary results on enriched coneds}

We will need the following simple, yet very useful proposition. We say that a bifunctor $\odot \colon \mcP \otimes \mcQ \to \mcN$ is \emph{balanced} if there is provided a natural isomorphism $\beta \colon (P \otimes K) \odot Q \cong P \odot (K \otimes Q)$ satisfying certain obvious relation with regard to the associativity isomorphisms $\alpha$ for the actions of $\mcC$ on $\mcP$ and $\mcQ$:
\[\xymatrix@R=1pc{
(P \otimes (K \otimes L)) \odot Q \ar@{<->}[r]_-\cong^-\alpha \ar@{<->}[dd]^-\cong_-\beta &
((P \otimes K) \otimes L) \odot Q \ar@{<->}[rd]_-\cong^-\beta \\
& & (P \otimes K) \odot (L \otimes Q) \\
P \odot ((K \otimes L) \otimes Q) \ar@{<->}[r]^-\cong_-\alpha &
P \odot (K \otimes (L \otimes Q)) \ar@{<->}[ru]^-\cong_-\beta
}\]
and also the unit isomorphisms (a commutative triangle). In particular, each action $\mcV \otimes \mcM \to \mcM$ is balanced.

\begin{proposition}\label{p:coend_projective}
Let $\odot \colon \mcP \otimes \mcQ \ra \mcN$ be a left closed bifunctor which preserves colimits in the second variable and is balanced. Then the bifunctor
\[\odot_\mcC \colon [\mcC^\op,\mcP]_\pt \otimes[\mcC,\mcQ] \ra \mcN\]
is also left closed with respect to the indicated pointwise and projective cofibrations.

More generally, if  $V \to W$ is a pointwise strong cofibration and $X \to Y$ a cofibration between pointwise cofibrant diagrams such that the coends $V \odot_\mcC X$ and $W \odot_\mcC X$ exist and are good, then so are $V \odot_\mcC Y$ and $W \odot_\mcC Y$ and in the square
\[\xymatrix{
V \odot_\mcC X \ar[r] \ar@{ >->}[d] & W \odot_\mcC X \ar@{ >->}[d] \\
V \odot_\mcC Y \ar[r] & W \odot_\mcC Y \\
}\]
both vertical maps and the pushout corner map are cofibrations.
\end{proposition}

\begin{proof}
Let $X \to Y$ be a transfinite composite of a chain $Y_n$ with $X = Y_0$ and with $Y_{n-1} \to Y_n$ a pushout of a generating cofibration $\mcC_c\otimes M\ra\mcC_c\otimes N$ in $[\mcC,\mcM]$; we may assume that $M = Y_{n-1}c$ and is therefore cofibrant. Then, for this generating cofibration, we have an isomorphism
\[\xymatrix{
V \odot_\mcC (\mcC_c \otimes M) \ar[r] \ar[d] &
W \odot_\mcC (\mcC_c \otimes M) \ar[d] &
\ar@{}[d]|-*{\cong} &
Vc \odot M \ar@{ >->}[r] \ar@{ >->}[d] &
Wc \odot M \ar@{ >->}[d] \\
V \odot_\mcC (\mcC_c \otimes N) \ar[r] &
W \odot_\mcC (\mcC_c \otimes N) &
&
Vc \odot N \ar@{ >->}[r] &
Wc \odot N
}\]
and the pushout corner map is a cofibration since both $Vc \to Wc$ and $M \to N$ are strong cofibrations. Therefore, in the following horizontal pushout of the above square
\[\xymatrix{
V \odot_\mcC Y_{n-1} \ar[r] \ar@{ >->}[d] & W \odot_\mcC Y_{n-1} \ar@{ >->}[d] \\
V \odot_\mcC Y_n \ar[r] & W \odot_\mcC Y_n
}\]
the pushout corner map is also a cofibration. The same holds for the vertical transfinite composite -- this is the square from the statement.
\end{proof}

This proposition has the following special cases, namely, the weighted colimit and the left Kan extension.

\begin{proposition}\label{p:weighted_colimit_cofibrations}
Let $\mcC$ be an arbitrary small $\mcV$-category. Then the (partially defined) weighted colimit bifunctor $\otimes_\mcC \colon [\mcC^\op,\mcV]_\pt \otimes [\mcC,\mcM] \to \mcM$ is left closed.

More generally, if $V \to W$ is a pointwise flat map and $X \to Y$ a cofibration between pointwise cofibrant diagrams such that the weighted colimits $V \otimes_\mcC X$ and $W \otimes_\mcC X$ exist and are good, then so are $V \otimes_\mcC Y$ and $W \otimes_\mcC Y$ and in the square
\[\xymatrix{
V \otimes_\mcC X \ar[r] \ar@{ >->}[d] & W \otimes_\mcC X \ar@{ >->}[d] \\
V \otimes_\mcC Y \ar[r] & W \otimes_\mcC Y \\
}\]
both vertical maps and the pushout corner map are cofibrations.
\end{proposition}

\begin{proof}
This is a sepcial case of Proposition~\ref{p:coend_projective} for the action $\otimes \colon \mcV \otimes \mcM \to \mcM$.
\end{proof}

A similar result holds in the opposite situation $[\mcC^\op,\mcV] \otimes [\mcC,\mcM]_\pt \to \mcM$. There is, however, a significant difference between these two statements. While the above is not applicable to trivial cofibrations (we do not know whether they are generated in the same way as cofibrations), the case of pointwise trivial cofibrations also follows from Proposition~\ref{p:coend_projective}.

\begin{proposition}\label{p:left_Kan_extension_cofibrations}
Let $F \colon \mcC \to \mcD$ be a $\mcV$-functor between small $\mcV$-categories. Then the left Kan extension functor $F_! \colon [\mcC,\mcM] \to [\mcD,\mcM]$ is left closed.
\end{proposition}

\begin{proof}
This is a sepcial case of Proposition~\ref{p:coend_projective} for the action $\otimes \colon [\mcD,\mcV] \otimes \mcM \to [\mcD,\mcM]$; in the resulting left closed bifunctor
\[\otimes_\mcC \colon [\mcC^\op,[\mcD,\mcV]]_\pt \otimes [\mcC,\mcM] \to [\mcD,\mcM],\]
we fix as the first argument the cofibrant object $\mcD(F-,-) \in [\mcC^\op,[\mcD,\mcV]]_\pt$.
\end{proof}

\subsection*{Cubical diagrams and their pushout corner maps}\label{s:cubical_diagrams}

For a finite set $S$, we consider the category $\mcP(S)$ of all subsets of $S$ and the category $\mcP_0(S)$ of all proper subsets. We call any diagram $X\col\mcP(S)\ra\mcM$ an \emph{$S$-cubical diagram}. We denote the left Kan extension of its restriction to $\mcP_0(S)$ by $X_0$. Since $\mcP_0(S)$ is full, $X_0 I = X I$ for $I\in\mcP_0(S)$, and $X_0 S = \colim X|_{\mcP_0(S)}$. The \emph{pushout corner map} of $X$ is the canonical map
\[\pcm X\col X_0 S \lra X S.\]

Suppose that $S=S_1+S_2$, a disjoint union. We may then compute the pushout corner map for $X$ by considering the pushout corner maps in the two ``directions'' $S_1$ and $S_2$. Concretely,
\[\mcP(S_1+S_2)\cong\mcP(S_1)\times\mcP(S_2)\]
and denote by $X_1$ consider the left Kan extension of the restriction to $\mcP_0(S_1)\times\mcP(S_2)$, by $X_2$ the left Kan extension of the restriction to $\mcP(S_1)\times\mcP_0(S_2)$ and by $X_{12}$ the one for $\mcP_0(S_1)\times\mcP_0(S_2)$. We obtain a square
\[\xymatrix{
X_{12} S \ar[r] \ar[d] & X_1 S \ar[d] \\
X_2 S \ar[r] & X S
}\]
We will see that its pushout corner map is canonically isomorphic to $\pcm X$. We have $X_1 S = \colim X|_{\mcP_0(S_1)\times\mcP(S_2)}$ and similarly for the remaining corners. The pushout is therefore isomorphic to the colimit of the restriction to the union
\[(\mcP_0(S_1)\times\mcP(S_2))+_{(\mcP_0(S_1)\times\mcP_0(S_2))}(\mcP(S_1)\times\mcP_0(S_2))=\mcP_0(S)\]
and finally, the pushout corner map is $X_0 S \to X S$, as claimed.

If $X$ is an $S$-cubical diagram in $\mcM$, $Y$ is a $T$-cubical diagram in $\mcN$ and there is given a bifunctor $-\odot-\col\mcM\otimes\mcN\ra\mcQ$ then we denote by $X\odot Y$ the resulting ($S+T$)-cubical diagram in $\mcQ$ which has, for $I\in\mcP(S)$ and $J\in\mcP(T)$,
\[(X\odot Y)(I+J)=X I \odot Y J.\]
Under the assumption that the bifunctor $-\odot-$ is cocontinuous in each variable, the above square becomes
\[\xymatrix@C=50pt{
X_0 S \odot Y_0 T \ar[r]^{\id\odot(\pcm Y)} \ar[d]_{(\pcm X)\odot\id} & X_0 S \odot Y T \ar[d]^{(\pcm X)\odot\id} \\
X S \odot Y_0 T \ar[r]_{\id\odot(\pcm Y)} & X S \odot Y T
}\]
and we may conclude that
\[\pcm(X\odot Y)\cong\pcm((\pcm X)\odot(\pcm Y)).\]

\section{Enriched Reedy categories}

\subsection*{Reedy categories}

A small $\mcV$-category $\mcD$ is called \emph{direct} if there is given a labeling $|{-}| \colon \mcD\ra\lambda$ of objects of $\mcD$ by an ordinal $\lambda$ in such a way that $\mcD(d',d)=0$ if $|d'|\geq|d|$ with the sole exception $d'=d$ where we require that $\id_d \colon \bbone\ra\mcD(d,d)$ is an isomorphism. We call $|d|$ the \emph{degree} of the object $d$. The dual notion is that of an \emph{inverse} category -- the objects are still labelled by an ordinal but the morphisms point downwards.

A small $\mcV$-category $\mcR$ is called \emph{Reedy} if there is given a degree function $|{-}| \colon \mcR\ra\lambda$ and two subcategories $\mcR_+$ and $\mcR_-$ with the same sets of objects as $\mcR$ and such that
\begin{itemize}
\item the category $\mcR_+$ is direct (with the given degree function),
\item the category $\mcR_-$ is inverse (with the given degree function) and
\item for each $c',\, c\in\mcR$, the natural map
\begin{equation}\label{eq:Reedy_decomposition}
\sum_{d\in\mcR} \mcR_+(d,c) \otimes \mcR_-(c',d) \xra\cong \mcR(c',c),
\end{equation}
given by the inclusion into $\mcR$ and composing, is an isomorphism.
\end{itemize}

For an object $c\in\mcR$, we define the subdiagram $\partial\mcR^c\subseteq\mcR^c=\mcR(-,c)$ given by the sub-coproduct of \eqref{eq:Reedy_decomposition} running over all $d\in\mcR$ different from $c$ (one might restrict to those of degree less than $|c|$). The corresponding weighted colimit, for a diagram $X \colon \mcR \to \mcM$, is called the \emph{latching object} of $X$ at $c$ and denoted $L_cX=\partial\mcR^c\otimes_\mcR X$; in general, it may fail to exist. The inclusion $\partial\mcR^c \to \mcR^c$ will be denoted by $j^c$. Dually, we define the inclusion $j_{c'} \colon \partial\mcR_{c'} \to \mcR_{c'}=\mcR(c',-)$ whose domain is the sub-coproduct of \eqref{eq:Reedy_decomposition} running over all $d\in\mcR$ different from $c'$.

We denote by $\mcR_{(n)}$ the full subcategory of all objects of degree $n$; $\mcR_{(\leq n)}$ and $\mcR_{(<n)}$ are defined similarly. We define the \emph{$n$-skeleton}%
\footnote{
	For another point of view, as in \cite{GoerssJardine}, denote by $i_n \colon \mcR_{(\leq n)} \to \mcR$ the inclusion. The $n$-skeleton functor $\sk_n$ is the composition $\sk_n=(i_n)_!(i_n)^*$ using the induced restriction $(i_n)^*$ and left Kan extension $(i_n)_!$ functors
	\[(i_n)_! \col [\mcR_{(\leq n)},\mcM]\rightleftarrows[\mcR,\mcM] \loc (i_n)^*.\]
}
$\sk_n X = (\sk_n \mcR(-,-)) \otimes_\mcR X$, where $\sk_n \mcR(-,-)$ consists of those morphisms that factor through an object of degree at most $n$, i.e.
\[\sk_n \mcR(c',c) = \sum_{d\in\mcR_{(\leq n)}} \mcR_+(d,c) \otimes \mcR_-(c',d).\]
For $c \in \mcR_{(n)}$, one has $(\sk_nX)c\cong Xc$ and $(\sk_{n-1}X)c \cong L_cX$.

There is a way of building $X\in[\mcR,\mcM]$ inductively. Namely, $X\cong\colim_n\sk_nX$ and, for each $n$, there is a square
\begin{equation}\label{eq:skeleton_construction}
\xymatrixc{
\topbox{\sum\limits_{c\in\mcR_{(n)}}{}}{\mcR_c \otimes L_cX +_{\partial\mcR_c \otimes L_cX} \partial\mcR_c \otimes Xc}{} \ar[r] \ar[d] & \sk_{n-1} X \ar[d] \\
\topbox{\sum\limits_{c\in\mcR_{(n)}}{}}{\mcR_c \otimes Xc}{} \ar[r] & \sk_n X \po
}\end{equation}
where the component of the top map at $c$ is the adjunct of the canonical isomorphism $L_cX\ra(\sk_{n-1}X)c$ on the first summand and the canonical map $\partial\mcR_c \otimes Xc = \sk_{n-1}\mcR(c,-) \otimes Xc \to \sk_{n-1}X$ on the second summand. We will now show that this square is a pushout square. Observe that the square is obtained as the weighted colimit of $X$ with weights forming the following diagram
\[\xymatrix{
\topbox{\sum\limits_{c\in\mcR_{(n)}}{}}{\mcR_c \otimes \partial\mcR^c +_{\partial\mcR_c \otimes \partial\mcR^c} \partial\mcR_c \otimes \mcR^c}{} \ar[r] \ar[d] & \sk_{n-1} \mcR(-,-) \ar[d] \\
\topbox{\sum\limits_{c\in\mcR_{(n)}}{}}{\mcR_c \otimes \mcR^c}{} \ar[r] & \sk_n \mcR(-,-) \po
}\]
They form a pushout square by an easy inspection: the map on the left is an inclusion of a sub-coproduct with complementary summands exactly $\sum_{c\in\mcR_{(n)}} (\mcR_+)_c \otimes (\mcR_-)^c$ and the same is true for the map on the right.

We note that the pushout square \eqref{eq:skeleton_construction} considerably simplifies for a direct category since then $\partial\mcR_c=0$ and thus, the top left corner reads $\sum_{c\in\mcR_{(n)}} \mcR_c \otimes L_cX$.

Clearly, in the inductive construction of $X$ from $\sk_nX$, we may skip all the steps in which the latching map is an isomorphism. This will be important in the relative case to follow, together with the following lemma.

\begin{lemma}\label{l:cells}
Let $\mcR$ be a locally flat Reedy $\mcV$-category. Then $\pcm(j^c \otimes_\mcR j_{c'})$ is flat.
\end{lemma}

\begin{proof}
The square $j^c \otimes_\mcR j_{c'}$ is isomorphic to
\[\xymatrix{
\partial\mcR^c \otimes_\mcR \partial\mcR_{c'} \ar[r] \ar[d] & \mcR^c \otimes_\mcR \partial\mcR_{c'} \ar[d] & \ar@{}[d]|-*{\cong} & \mcR_\nmono(c',c) \ar[r] \ar[d] & \mcR_\npos(c',c) \ar[d] \\
\partial\mcR^c \otimes_\mcR \mcR_{c'} \ar[r] & \mcR^c \otimes_\mcR \mcR_{c'} & & \mcR_\nneg(c',c) \ar[r] & \mcR(c',c)
}\]
where $\mcR_\npos(c',c)$ denotes the sub-coproduct in \eqref{eq:Reedy_decomposition} indexed by all $d$ with $d \neq c'$, similarly $\mcR_\nneg(c',c)$ has $d \neq c$ and $\mcR_\nmono(c',c)$ has $d \neq c'$ and $d \neq c$: all the isomorphisms are instances of the Yoneda lemma with the exception of
\[\partial\mcR^c \otimes_\mcR \partial\mcR_{c'} \cong \mcR_\nmono(c',c) = \sum_{\makebox[0pt]{$\scriptstyle d\in\mcR,\,d \neq c' \text{ and } d \neq c$}} \mcR_+(d,c) \otimes \mcR_-(c',d).\]
Here, the map from the left to the right is simply the composition and its inverse is easily constructed on each summand from the inclusions $\mcR_+(d,c) \subseteq \partial\mcR^c d$ and $\mcR_-(c',d) \subseteq \partial\mcR_{c'} d$.

The pushout in the right-hand square is $\mcR_\nd(c',c)$, the sub-coproduct containing all the summands with the exception of the identity $c'=d=c$. The pushout corner map is thus either an isomorphism (when $c' \neq c$) or a pushout of $0 \to \bbone$ (when $c'=c$); in both cases, it is flat.
\end{proof}

\subsection*{Relative Reedy categories}

For technical reasons, namely to deal with Reedy cofibrations between diagrams that are not necessarily Reedy cofibrant, we will need to cover the relative situation.

Let $\mcD$ be a direct $\mcV$-category. We say that $\mcD' \subseteq \mcD$ is a \emph{(full) direct subcategory} if it is a full subcategory such that, for $d \in \mcD \smallsetminus \mcD'$ and $d' \in \mcD'$, we have $\mcD(d,d') = 0$ (i.e.\ $\mcD'$ forms an ``initial part'' of $\mcD$). Dually, a \emph{(full) inverse subcategory} $\mcI' \subseteq \mcI$ of an inverse $\mcV$-category $\mcI$ is a full subcategory such that, for $i' \in \mcI'$ and $i \in \mcI \smallsetminus \mcI'$, we have $\mcI(i',i) = 0$.

Let $\mcR$ be a Reedy $\mcV$-category. We say that $\mcR' \subseteq \mcR$ is a \emph{(full) Reedy subcategory} if $\mcR'_+ \subseteq \mcR_+$ is a direct subcategory and $\mcR'_- \subseteq \mcR_-$ an inverse subcategory; this implies that $\mcR' \subseteq \mcR$ is indeed full.

In this situation, we say that a diagram $X \colon \mcR \to \mcM$ is \emph{Reedy cofibrant away from $\mcR'$} if, for each $c \in \mcR \smallsetminus \mcR'$, the latching object $L_c X$ exists and the respective latching map $L_c X \to Xc$ is a cofibration; in particular, both $L_c X$ and $Xc$ must be good.

For the following lemma, we denote $\iota \colon \mcR' \to \mcR$ the inclusion.

\begin{lemma}\label{l:cofibrations_characterization}
Let $X \colon \mcR \to \mcM$ be a diagram such that $\iota_! \iota^* X$ exists and is pointwise good away from $\mcR'$. Then $X$ is Reedy cofibrant away from $\mcR'$ if and only if the canonical map $\iota_! \iota^* X \to X$ lies in the cellular closure of
\[\{\pcm(j_d \otimes k) \mid \textrm{$d \in \mcR \smallsetminus \mcR'$ and $k$ a cofibration in $\mcM$}\}\]
(closure under coproducts, pushouts and transfinite composites).
\end{lemma}

\begin{proof}
For the purpose of this proof, we define a Reedy cofibration away from $\mcR'$ to be a map $f \colon X \to Y$ such that for all $c \in \mcR \smallsetminus \mcR'$, both $L_c X$ and $L_c Y$ exist, the induced map $L_c X \to L_c Y$ is a cofibration and the pushout corner map in
\[\xymatrix{
L_cX \ar[r] \ar[d] & L_cY \ar[d] \\
Xc \ar[r] & Yc
}\]
is a cofibration. Later, we will present a more general definition.

First observe that $\iota_! \iota^* X$ is cofibrant away from $\mcR'$ -- Lemma~\ref{l:coend_restriction_extension} gives the two right-hand side isomorphisms in the diagram
\[\xymatrix@C=0.5pc{
L_c (\iota_! \iota^* X) \ar@{}[r]|-{=} \ar[d] & \partial\mcR^c \otimes_\mcR \iota_! \iota^* X \ar@{}[r]|-{\cong} \ar[d] & \iota^* \partial\mcR^c \otimes_\mcR \iota^* X \ar[d]^-{=} \\
(\iota_! \iota^* X) c \ar@{}[r]|-{\cong} & \mcR^c \otimes_\mcR \iota_! \iota^* X \ar@{}[r]|-{\cong} & \iota^* \mcR^c \otimes_\mcR \iota^* X
}\]
and $\iota^* \partial\mcR^c = \iota^* \mcR^c$ for $c \in \mcR \smallsetminus \mcR'$ (there are no morphisms in $\mcR_-$ from $\mcR'$ to $\mcR \smallsetminus \mcR'$).

The square of latching maps for the generating map is $\pcm(j^c \otimes_\mcR j_d) \otimes k$. Since $\pcm(j^c \otimes_\mcR j_d)$ is flat by Lemma~\ref{l:cells}, $\pcm(j_d \otimes k)$ is a Reedy cofibration away from $\mcR'$ and so is its arbitrary vertical pushout. This implies that any map from the concerned closure whose domain is Reedy cofibrant away from $\mcR'$ will have codomain also Reedy cofibrant away from $\mcR'$.

The necessity follows from building $X$ from $\iota_! \iota^* X$ in the canonical way as in \eqref{eq:skeleton_construction}. Since these diagrams agree at $\mcR'$, one only needs cells for objects $c \in \mcR \smallsetminus \mcR'$, for which the corresponding latching maps $L_cX \to Xc$ are assumed to be cofibrations.
\end{proof}

\subsection*{Reedy cofibrations}

Consider a transformation $f \colon X \to Y$ and make it into a diagram $F \colon \mcR \otimes [1] \to \mcM$. We say that $f$ is a \emph{Reedy cofibration} if $F$ is cofibrant relative to $\mcR \otimes \{0\}$. A rather simple observation is that good diagrams are exactly the pointwise good (the latching maps of the diagram $\Id_X$ corresponding to the identity $\id_X$ are the identity maps $\id_{Xc}$).

Clearly, if both $L_c X$ and $L_c Y$ exist, then $L_{c \otimes 1}F \to F(c \otimes 1)$ is the pushout corner map in
\[\xymatrix{
L_cX \ar[r] \ar[d] & L_cY \ar[d] \\
Xc \ar[r] & Yc
}\]
and thus, our definition generalizes the one used in the proof of Lemma~\ref{l:cofibrations_characterization}.

Applying Lemma~\ref{l:cofibrations_characterization} to this particular situation, we get the following proposition.

\begin{proposition}\label{prop_projective_cofibrations_over_direct_categories}
Let $f \colon X \to Y$ be a transformation with a pointwise good domain $X$. Then $f$ is a Reedy cofibration if and only if it lies in the cellular closure of
\[\{\pcm(j_c \otimes k) \mid \textrm{$c \in \mcR$ and $k$ a cofibration in $\mcM$}\}\]
(closure under coproducts, pushouts and transfinite composites).
\end{proposition}

\begin{proof}
In exactly the same way $F$ is built from $\Id_X$ using the cells $\pcm(j_{c \otimes 1} \otimes k)$ from Lemma~\ref{l:cofibrations_characterization}, the codomain $Y$ of $F$ is built from the codomain $X$ of $\Id_X$ using the cells $\pcm(j_c \otimes k)$.
\end{proof}

\begin{proposition}\label{p:coend_Reedy}
Let ${-} \otimes {-} \colon \mcP \otimes \mcQ \ra \mcN$ be a left closed bifunctor which preserves colimits in both variables and is balanced. Then the bifunctor
\[{-} \otimes_\mcR {-} \colon [\mcR^\op,\mcP]_\Reedy \otimes [\mcR,\mcQ]_\Reedy \ra \mcN\]
is left closed with respect to the Reedy cofibrations.
\end{proposition}

\begin{proof}
Again, this is about weights. Let us consider the generating Reedy cofibrations $\pcm(k \otimes j^c)$, $\pcm(j_{c'} \otimes l)$ and the pushout corner map of their image
\[\pcm(\pcm(k \otimes j^c) \otimes_\mcR \pcm(j_{c'} \otimes l)) \cong \pcm(k \otimes (j^c \otimes_\mcR j_{c'}) \otimes l).\]
According to Lemma~\ref{l:cells}, $\pcm(j^c \otimes_\mcR j_{c'})$ is flat and this finishes the proof.
\end{proof}

\section{Cofibration categories}

Finally, we introduce the main notion of this paper, details may be found in \cite{Radulescu-Banu}. We keep our assumption that $\mcM$ should contain an initial object and that all functors should preserve it.

A \emph{cofibration category} is a category equipped with subcategories $\sfC$ of \emph{cofibrations} and $\sfW$ of \emph{weak equivalences} satisfying
\begin{itemize}
\item Weak equivalences satisfy 2-out-of-3 property and include all isomorphisms.
\item The classes $\sfC$ of cofibrations and $\sfW\cap\sfC$ of trivial cofibrations are left closed. Observe that they share the same class of objects and, again, we call these objects \emph{good}.
\item Every map with a good domain admits a (cofibration, weak equivalence) factorization.
\end{itemize}
In the proceeding, we will always assume that all good objects are cofibrant; they are however rarely cofibrant with respect to $\sfW\cap\sfC$ so that the previous generality was useful.

We say that a functor $F \colon \mcM \to \mcN$ is \emph{left Quillen} if it preserves cofibrations and weak equivalences between cofibrant objects. It follows from Brown's lemma \cite[Lemma 1.3.1]{Radulescu-Banu} that this is equivalent to $F$ preserving cofibrations and trivial cofibrations, i.e.\ that $\mcM^{[1]}\ra\mcN^{[1]}$ preserves cofibrant objects and weak equivalences between them. A bifunctor $\mcP \otimes \mcQ \ra \mcN$ is called a \emph{left Quillen bifunctor} if the induced $\mcP^{[1]}\otimes\mcQ^{[1]}\ra\mcN^{[1]\times[1]}$ preserves cofibrant objects and weak equivalences between them.

Let $\mcV$ be a cocomplete monoidal category. We say that $\mcM$ is a \emph{cofibration $\mcV$-category} if it is a cofibration category equipped with an action $\otimes \colon \mcV\otimes\mcM\ra\mcM$ (i.e.\ we do not require $\mcM$ to be enriched but rather tensored).

In this situation, we say that a map of $\mcV$ is \emph{flat} (with respect to this action) if the induced functor $i\otimes- \colon \mcM\ra\mcM^{[1]}$ is left Quillen. This implies that both $K \otimes -$ and $L \otimes -$ are left Quillen and thus, $K$ and $L$ are ``flat objects''. Suppose now that $\mcV$ is suitably cocomplete and that the action preserves colimits in the first variable. Then $\mcV$ together with flat maps becomes a category with cofibrations which is reduced and both
\[\otimes \colon \mcV \otimes \mcV \ra \mcV \quad \textrm{and} \quad \otimes \colon \mcV \otimes \mcM \ra \mcM,\]
the monoidal structure and the action, are left closed/left Quillen bifunctors. Regardless of $\mcM$, the canonical map $0\ra\bbone$ is always flat (and, more generally, the unit $\Set\ra\mcV$ takes monos to flats).

We say that a $\mcV$-category $\mcC$ is \emph{locally flat} if all the hom-objects $\mcC(c',c) \in \mcV$ are flat. In the case $\mcV=\Set$, all injective maps are flat and thus, an ordinary category is always locally flat and the same is true for the $\sSet$-enrichement. For general $\mcV$, not every $\mcV$-category is locally flat. This was our example in the introduction, where we were in fact talking about the $\Ch_\bbZ$-category with one object and endomorphisms forming $\bbZ/2$ (the full subcategory of $\Ch_\bbZ$ on the object $\bbZ/2$).

\section{Proof of Theorem~\ref{t:projective_structure}}\label{s:proof_main}

Let $\mcD$ be a small locally flat direct $\mcV$-category and $\mcD' \subseteq \mcD$ a direct subcategory. We say that a transformation $f \colon X \to Y$ is a \emph{weak equivalence away from $\mcD'$} if it is a pointwise weak equivalence whose restriction to $\mcD'$ is an isomorphism.

We consider the category of diagrams $X \colon \mcD \to \mcM$ with cofibrations away from $\mcD'$ and weak equivalences away from $\mcD'$. The resulting cofibrant replacement will be addressed as a \emph{cofibrant replacement away from $\mcD'$}.

The following example will be important in translating the factorization property into a cofibrant replacement that is technically easier to handle.

\begin{example}\label{e:arrow_category}
The case of $\mcD = [1]$ with $\mcD' = \{0\}$. Then the objects of $[\mcD,\mcM]$ are maps in $\mcM$. An object is cofibrant away from $\mcD'$ if and only if it is a cofibration. A cofibrant replacement of $f$ away from $\mcD'$ is therefore a factorization of $f$ into a cofibration $gh^{-1}$ followed by a weak equivalence $w$
\[\xymatrix{
\bullet \ar[r]^-h_-\cong \ar@{ >->}[d]_-g & \bullet \ar[d]^-f & & 0 \ar[d] \\
\bullet \ar[r]_-w^-\sim & \bullet & & 1
}\]
We observe that the existence of such a factorization requires the domain of $f$ to be good.

Another interesting example is that of $\mcD = \mcP(S)$ with $\mcD' = \mcP_0(S)$. Then a diagram $X \colon \mcD \to \mcM$ is a cube in $\mcM$. It is cofibrant away from $\mcD'$ if and only if the pushout corner exists and the pushout corner map is a cofibration. In particular, a diagram may only admit a cofibrant replacement if its pushout corner exists and is good.
\end{example}

\subsection*{Proof for direct categories}

All axioms are verified pointwise except for the factorization. For a map $f \colon X \to Y$ with $X$ pointwise good, the factorization is obtained, as in Example~\ref{e:arrow_category}, as a cofibrant replacement of the corresponding diagram $F \colon \mcD \otimes [1] \to \mcM$ away from $\mcD \otimes \{0\}$.

More generally, let $\mcC$ be a small locally flat direct $\mcV$-category and $\mcC'$ a direct subcategory; we denote by $\iota \colon \mcC' \to \mcC$ the inclusion. Let $F \colon \mcC \to \mcM$ be a diagram such that $\iota_! \iota^* F$ exists and is pointwise good away from $\mcC'$. We will produce a cofibrant replacement $G$ of $F$ away from $\mcC'$ as a colimit of a certain chain $G_n$; this will be constructed inductively together with a cocone $\lambda_n \colon G_n \to F$ in such a way that
\begin{itemize}
\item
	$G_{-1} = \iota_! \iota^* F$,
\item
	each $G_{n-1} \to G_n$ is a pushout of a coproduct of generating cofibrations $\mcC_c \otimes g_c$ for $c \in \mcC \smallsetminus \mcC'$ with $|c| = n$,
\item
	$\lambda_n \colon G_n \to F$ is a weak equivalence at all objects of degree $\leq n$.
\end{itemize}
This implies easily that the colimit $G = \colim G_n$ is indeed a cofibrant replacement of $F$ away from $\mcC'$ (since the maps $G_{k-1} \to G_k$ are isomorphisms at $\mcC_{(<k)}$, the component of $\lambda \colon G \to F$ at $c \in \mcC_{(n)}$ is essentially equal to $\lambda_n$ and thus a weak equivalence).

In the induction step, $G_{n-1}$ is cofibrant by Lemma~\ref{l:cofibrations_characterization} and thus, $G_{n-1}c$ is good. Thus, we may factor $G_{n-1}c \to Fc$ as
\[\xymatrix{
G_{n-1}c \ar@{ >->}[r]^-{g_c} & Gc \ar[r]^-{h_c}_-\sim & Fc.
}\]
Then we use the cofibrations $g_c$ to attach cells to $G_{n-1}$,
\[\xymatrix{
\topbox{\sum\limits_{c\in\mcC_{(n)}}{}}{\mcC_c \otimes G_{n-1}c}{} \ar[r] \ar[d] & G_{n-1} \ar[d] \\
\topbox{\sum\limits_{c\in\mcC_{(n)}}{}}{\mcC_c \otimes Gc}{} \ar[r] & G_n \po
}\]
to construct the object $G_n$; it admits an obvious map to $F$. At objects of degree $<n$, the diagrams on the left are zero and thus, the map on the right is an isomorphism; consequently, $G_n \to F$ is a weak equivalence at $\mcC_{(<n)}$. At objects of degree $n$, the top map is an isomorphism, hence also the bottom map; consequently, the map $G_nc \to Fc$ is essentially an inverse of $h_c$ and thus a weak equivalence at each $c \in \mcC_{(n)}$.\qed

\vskip 2\topsep

The difficulty of proving Theorem~\ref{t:projective_structure} for a general $\mcV$-category lies in the lack of a characterization as in Proposition~\ref{prop_projective_cofibrations_over_direct_categories}.

\subsection*{Proof for general categories}
Our strategy is to ``replace'' $\mcC$ by a direct category, use the factorization there and push the result back to $\mcC$. We define this ``direct replacement'' $\Delta\mcC$ of $\mcC$ in the following way. The objects of $\Delta\mcC$ are non-empty finite sequences $(c_0,\ldots,c_n)$ of objects of $\mcC$ and
\[\Delta\mcC\Big((c_0',\ldots,c_k'),(c_0,\ldots,c_n)\Big)=\Big(\sum_{\varphi(k)=n}\bbone\Big) + \Big(\sum_{\varphi(k)\neq n}\mcC(c'_k,c_n)\Big)\]
where the coproducts run over all injective monotone $\varphi \colon [k] \to [n]$ with $c'_i=c_{\varphi(i)}$ for all $i=0,\ldots,k$. This is a direct category with the degree function $(c_0,\ldots,c_n)\mapsto n$.

There is a canonical $\mcV$-functor $P \colon \Delta\mcC\ra\mcC$ sending $(c_0,\ldots,c_n)\mapsto c_n$ and which has the obvious effect on morphisms (in the case $\varphi(k)=n$, hence $c_k'=c_n$, the map $\bbone \to \mcC(c_k',c_n)$ is the unit map). To produce a factorization in $[\mcC,\mcM]$, we use the cofibration structure on $[\Delta\mcC,\mcM]$. Let $f \colon X\ra Y$ be a map in $[\mcC,\mcM]$ and let
\[\xymatrix{
P^*X \ar@{ >->}[r]^-g & Z \ar[r]^-h_-\sim & P^*Y
}\]
be a factorization of $P^*f$ in $[\Delta\mcC,\mcM]$. We apply the left Kan extension functor $P_!$ to obtain
\[\xymatrix{
P_!P^*X \ar[r]^-{P_!g} \ar[d]_-{\varepsilon_X} & P_!Z \ar[r]^-{P_!h} & P_!P^*Y \ar[d]^-{\varepsilon_Y} \\
X \ar[rr]_-f & & Y
}\]
This yields the required factorization provided that
\begin{itemize}
\item
	the counit maps $\varepsilon_X$, $\varepsilon_Y$ are isomorphisms,
\item
	the map $P_!g$ is a cofibration,
\item
	the map $P_!h$ is a weak equivalence.
\end{itemize}

It is known that the first condition is equivalent to the right adjoint $P^*$ being fully faithful. From a transformation $P^*Z' \to P^*Z$, one reconstructs the transformation $Z' \to Z$ in the following way: its components are $Z'c = P^*Z'(c) \to P^*Z(c) = Zc$; they are natural because $\mcC(c',c) \otimes Zc' \to Zc$ can be reconstructed as $\mcC(c',c) \otimes P^*Z(c') \to P^*Z(c',c) \xla\cong P^*Z(c)$. The second condition is the content of Proposition~\ref{p:left_Kan_extension_cofibrations}.

To prove the third claim, we will give an alternative description of the left Kan extension $P_!Z$, similar to the non-enriched description using comma categories. The role of the comma category is taken by an ordinary (non-enriched) category $\Delta_c\mcC$, whose objects are sequences of the form $(c_0,\ldots,c_{n-1},c)\in\Delta\mcC$ and whose morphisms are all injective monotone $\varphi \colon [k]\ra[n]$ satisfying $c'_i=c_{\varphi(i)}$. We will freely alternate between $\Delta_c\mcC$ and its associated $\mcV$-category. There is an obvious $\mcV$-functor $\Delta_c\mcC\ra\Delta\mcC$ that is the identity on morphisms $\varphi$ with $\varphi(k)=n$ and that sends $\varphi$ with $\varphi(k)\neq n$ to $\id_c\in\mcC(c,c)$ in the component corresponding to $\varphi$. For a diagram $Z \in [\Delta\mcC,\mcM]$, we denote by $Z_c$ its restriction to $\Delta_c\mcC$.

\begin{lemma}
There is a natural isomorphism $\colim Z_c \cong (P_! Z) c$.
\end{lemma}

\begin{proof}
It is easy to see that the action of $P$ on morphsisms induces a $\mcV$-natural isomorphism
\[\colim\nolimits_{\Delta_c\mcC}\Delta\mcC((c_0,\ldots,c_n),-) \xlra\cong \mcC(P(c_0,\ldots,c_n),c)\]
(the inverse is the inclusion $\mcC(c_n,c) \to \Delta\mcC((c_0,\ldots,c_n),(c_0,\ldots,c_n,c))$ of morphisms corresponding to the monotone map $d^{n+1} \colon [n] \to [n+1]$). When viewed as functors of $(c_0,\ldots,c_n)$, the weight on the left produces $\colim Z_c$ and the weight on the right $(P_!Z)c$.
\end{proof}

Thus, we have a diagram
\[\xymatrix{
Z(c) \ar[r]^-\sim \ar[d] & \leftbox{(P^*Y)(c)}{{}=Yc} \ar[d] \\
\rightbox{(P_!Z)c \cong {}}{\colim Z_c} \ar[r] & \leftbox{\colim (P^*Y)_c}{{} \cong (P_!P^*Y)c}
}\]
whose vertical maps are components of the corresponding colimit cocones and we are left to show that they are weak equivalences. The diagram $(P^*Y)_c$ is constant and the category $\Delta_c\mcC$ connected, which implies that, in fact, the right map is an isomorphism. Likewise, the diagram $Z_c$ is homotopy constant and it can be shown that $\Delta_c\mcC$ has a contractible nerve so that the canonical map into the \emph{homotopy} colimit is a weak equivalence, see \cite{ChacholskiScherer}. However, $Z$ is not cofibrant in general, and therefore, we choose a different approach.

We consider a subcategory $\Delta^\top_c\mcC \subseteq \Delta_c\mcC$ whose morphisms are the injective monotone $\varphi \colon [k]\ra[n]$ in $\Delta_c\mcC$ that preserve the top element, i.e.~$\varphi(k)=n$. The main technical advantage of $\Delta^\top_c\mcC$ over $\Delta_c\mcC$ is that it has an initial object $(c)$. Let us consider the following diagram.
\[\xymatrix{
\Delta_c\mcC \ar[rd]_-\rho \ar[rr]^-\Id & & \Delta_c\mcC \\
& \Delta^\top_c\mcC \ar@{c->}[ur]_-\iota \POS[];[u];**{}?(.4)*++!L{\scriptstyle\tau}+/-3pt/*\dir{=}+/6.5pt/*\dir{>}
}\]
The functor $\rho \colon \Delta_c\mcC\ra\Delta^\top_c\mcC$ adds ``$c$'' at the end, $(c_0,\ldots,c_{n-1},c) \mapsto (c_0,\ldots,c_{n-1},c,c)$. The indicated natural transformation $\tau \colon \Id\Ra\iota\rho$ has
\[\tau_{(c_0,\ldots,c_{n-1},c)} = d^{n+1} \colon (c_0,\ldots,c_{n-1},c) \to (c_0,\ldots,c_{n-1},c,c).\]
For a diagram $Z \in [\Delta\mcC,\mcM]$, we denote $Z^\top_c=Z|_{\Delta^\top_c\mcC}$.

We will now compare $\colim Z_c$ with the more manageable $\colim Z_c^\top$ using the following diagram
\[\xymatrix{
Z(c) \ar[r]^-{d^1} \ar[d] & Z(c,c) \ar[r]^-\id \ar[d] & Z(c,c) \ar[r]^-\id \ar[d] & Z(c,c) \ar[d] \\
\colim Z_c \ar[r]_-{\tau_*} & \colim (Z\rho)_c \ar[r]_-{\rho_*} & \colim Z^\top_c \ar[r]_-{\iota_*} & \colim Z_c
}\]
in which the composition across the bottom row is the identity while the composition across the top is $d^1$. Since $d^1 \colon Z(c)\ra Z(c,c)$ is a weak equivalence ($Z$ is weakly equivalent to $P^*Y$ that satisfies this), we observe that $Z(c)\ra\colim Z_c$ is, up to a weak equivalence, a retract of $Z(c,c)\ra\colim Z^\top_c$. Since weak equivalences are saturated, it is enough to prove that the latter map is a weak equivalence.

\begin{lemma}
Let $\mcD$ be an ordinary direct category with an initial object $d_0$ in degree 0 and $D \colon \mcD\ra\mcM$ a homotopy locally constant diagram (all morphisms in the diagram are weak equivalences) cofibrant away from $d_0$, i.e.\ such that $L_dD \to Dd$ is a cofibration for every $|d|>0$. Then each component $Dd\ra\colim D$ of the colimit cocone is a weak equivalence.
\end{lemma}

\begin{proof}
By induction, we may assume that each $L_dD \to Dd$ is a trivial cofibration for $|d|=n>0$, since the restriction of $D$ to each comma category $\mcD_{(<n)}/d$ also satisfies the assumptions of the lemma and $L_dD \cong \colim D|_{\mcD_{(<n)}/d}$. The induction step is summarized in the following pushout square
\[\xymatrixb{
\rightbox{\sum\limits_{d\in\mcD_{(n)}}}{L_dD} \ar[r] \POS[]*+{\vphantom{L_dD}} \ar@{ >->}[d]_-\sim & \colim \sk_{n-1}D \ar[d] \\
\rightbox{\sum\limits_{d\in\mcD_{(n)}}}{Dd} \ar[r] & \colim \sk_n D \po
}{
\rightbox{\sum\limits_{d\in\mcD_{(n)}}}{Dd} \ar[r] & \colim \sk_n D
}\qedhere\]
\end{proof}

We would like to apply this lemma to $Z^\top_c$. Since $Z_c^\top$ is weakly equivalent to $(P^*Y)_c^\top$, it is homotopy locally constant. The map $L_{(c_0,\ldots,c_n,c)} Z_c^\top \to Z_c^\top(c_0,\ldots,c_n,c)$ is obtained from $Z$ as a weighted colimit via a transformation of weights in $[\Delta\mcC^\op,\mcV]$ that will be described now.

Denoting the inclusion $j \colon \Delta_c^\top\mcC \to \Delta\mcC$, the second weight is simply $j_!((\Delta_c^\top\mcC)^{(c_0,\ldots,c_n,c)}) = \Delta\mcC^{(c_0,\ldots,c_n,c)}$ and consists of all morphisms $(c'_0,\ldots,c'_k) \to (c_0,\ldots,c_n,c)$. The first weight is given by the left Kan extension
\[j_!(\partial(\Delta_c^\top\mcC)^{(c_0,\ldots,c_n,c)}) \subseteq j_!((\Delta_c^\top\mcC)^{(c_0,\ldots,c_n,c)}) = \Delta\mcC^{(c_0,\ldots,c_n,c)}\]
and consists of exactly those morphisms $(c'_0,\ldots,c'_k) \to (c_0,\ldots,c_n,c)$ whose underlying monotone map $\varphi \colon [k] \to [n+1]$ misses at least one of $0,\ldots,n$.%
\footnote{
	The Yoneda isomorphism $\Delta_c^\top\mcC(-,(c_0,\ldots,c_n,c)) \otimes_{\Delta_c^\top\mcC} \Delta\mcC(?,j-) \xra\cong \Delta\mcC(?,(c_0,\ldots,c_n,c))$ (given by $j \otimes \id$ and composition) restricts to a map from $\partial\Delta_c^\top\mcC(-,(c_0,\ldots,c_n,c)) \otimes_{\Delta_c^\top\mcC} \Delta\mcC(?,j-)$ to the object described in the main text. The original map has an obvious pointwise inverse and this inverse restricts to an inverse.
}
Thus, the inclusion of weights is an inclusion of a sub-coproduct that misses some of the following summands: $\bbone$, corresponding to the identity, and $\mcC(c_n,c)$, corresponding to maps $(c_0,\ldots,c_n) \to (c_0,\ldots,c_n,c)$ that miss exactly the last object.

Consequently, the transformation is pointwise flat and, by Proposition~\ref{p:weighted_colimit_cofibrations}, the pushout corner map in
\[\xymatrix{
L_{(c_0,\ldots,c_n,c)} (P^*X)_c^\top \ar[r] \ar[d]_\cong & L_{(c_0,\ldots,c_n,c)} Z_c^\top \ar[d] \\
(P^*X)_c^\top(c_0,\ldots,c_n,c) \ar[r] & Z_c^\top(c_0,\ldots,c_n,c)
}\]
is a cofibration -- the hypotheses are satisfied since the map on the left is an isomorphism as indicated (because $(P^*X)_c^\top$ is constant). This also implies that the map on the right is the pushout corner map and is thus a cofibration as required.
\qed

\section{Proof of Theorem~\ref{t:weak_equivalences_generation}}

\subsection*{Proof of Theorem~\ref{t:weak_equivalences_generation}}

Let us denote the class from the statement by $\sfW_0$. The main idea of the proof is the following. For combinatorial model categories, trivial cofibrations are generated by the class
\[\{\mcC_c\otimes M\ra\mcC_c\otimes N\ |\ \textrm{$c\in\mcC$ and $M\ra N$ a trivial cofibration in $\mcM$}\}.\]
Since we suspect that this may be false for general cofibration categories, we will call the maps generated by this class \emph{good trivial cofibrations}. If $\mcC$ is locally flat, they are clearly trivial cofibrations and also belong to the class $\sfW_0$.

We construct a cofibrant replacement $\sfB\mcC_\bullet \otimes_{\Delta \otimes \mcC} \widetilde X$ of a pointwise cofibrant diagram $X$ in such a way that it turns pointwise trivial cofibrations to good trivial cofibrations. Finally, we prove that the augmentation $\sfB\mcC_\bullet \otimes_{\Delta \otimes \mcC} \widetilde X \to X$ belongs to $\sfW_0$ when $X$ is cofibrant.

The cofibrant replacement is produced from the (left Kan extension, restriction) adjunction
\[F \col [\delta\mcC,\mcM]\rightleftarrows[\mcC,\mcM] \loc U\]
with $\delta\mcC$ being a discrete category with the same set of objects. From this adjunction, we produce a comonad $FU$ and a ``bar construction'' functor
\[\sfB \colon [\mcC,\mcM]\ra\simp_\varepsilon[\mcC,\mcM]\]
into the category of augmented simplicial objects in $[\mcC,\mcM]$ with $\sfB_nX=(FU)^{n+1}X$ and whose augmentation $\sfB_0X=FUX\ra X=\sfB_{-1}X$ is the counit of the adjunction.

We apply this bar construction pointwise to the Yoneda embedding $\mcC_\bullet \colon \mcC^\op\ra[\mcC,\mcV]$ to obtain an augmented simplicial object $\sfB\mcC_\bullet \in \simp[\mcC^\op \otimes \mcC,\mcV]$ with
\[\sfB_n\mcC_\bullet = \sum_{c_0,\ldots,c_n\in\mcC} \mcC_{c_n} \otimes \mcC(c_0,\ldots,c_n) \otimes \mcC^{c_0}.\]
According to Proposition~\ref{p:BC_Reedy_cofibrant}, the diagram $\sfB\mcC_\bullet \in \simp[\mcC^\op \otimes \mcC,\mcV]$ is Reedy cofibrant.

Consider the pointwise tensor product
\[\otimes \colon [\mcC,\mcV] \otimes \mcM \lra [\mcC,\mcM]\]
that is clearly a left Quillen bifunctor; most importantly, this holds even if we equip the left hand-hand side with flat maps and (trivial) cofibrations and the right-hand side with (good trivial) cofibrations. Therefore, by Proposition~\ref{p:coend_Reedy}, the same is true for the induced functor
\[\otimes_{\Delta \otimes \mcC} \colon \simp[\mcC^\op\otimes\mcC,\mcV] \otimes \cosimp[\mcC,\mcM]_\pt \lra [\mcC,\mcM]\]
where the simplicial and cosimplicial directions are understood in the Reedy sense.

Let $\widetilde X \in \cosimp[\mcC,\mcM]_\pt$ be a frame on $X$ (see e.g.~\cite{Hovey} or \cite{Hirschhorn}). Our cofibrant replacement will be $\sfB\mcC_\bullet \otimes_{\Delta \otimes \mcC} \widetilde X$. Although Proposition~\ref{p:BC_c_Reedy_cofibrant} shows that the augmentation $\sfB\mcC_\bullet \otimes_{\Delta \otimes \mcC} \widetilde X \to X$ is a weak equivalence, we need to show that this map belongs to $\sfW_0$ and, to this end, we study this map for the particular diagrams appearing in the generating cofibrations, namely, when $X=\mcC_c\otimes M$. We consider a frame of the special form $\widetilde X=\mcC_c\otimes\widetilde M$, where $\widetilde M$ is a frame on $M$. For this particular choice, we get
\[\sfB\mcC_\bullet \otimes_{\Delta \otimes \mcC} \widetilde X \cong (\sfB\mcC_\bullet \otimes_\mcC \mcC_c) \otimes_\Delta \widetilde M \cong \sfB\mcC_c \otimes_\Delta \widetilde M\]
By Proposition~\ref{p:BC_c_Reedy_cofibrant}, $\sfB\mcC_c$ admits an extra degeneracy in such a way that the corresponding (augmented and extra degenerate) object is Reedy cofibrant. Thus, according to Proposition~\ref{p:extra_degenerate_realization}, the augmentation map
\[\sfB\mcC_c\otimes_\Delta\widetilde M\lra\mcC_c\otimes M\]
admits a section by a good trivial cofibration; in particular, this augmentation map lies in $\sfW_0$.

Now consider a strong cofibration $M \to N$ and, starting from a frame $\widetilde M$, construct a strong cofibration of frames $\widetilde M \to \widetilde N$ as in the following diagram
\[\xymatrix{
\Delta_0\cdot M \ar@{ >->}[r] \ar@{ >->}[d] & \widetilde M \ar[rr]^-\sim \ar@{ >->}[d] & & \const_\Delta M \ar[d] \\
\Delta_0\cdot N \ar@{ >->}[r] & \bullet \po \ar@{ >->}[r] & \widetilde N \ar[r]_-\sim & \const_\Delta N
}\]
(first form the indicated pushout and then factor the map from this pushout to $\const_\Delta N$ into a Reedy cofibration followed by a weak equivalence). We obtain a square
\[\xymatrix{
\sfB\mcC_c\otimes_\Delta\widetilde M \ar[r] \ar@{ >->}[d] & \mcC_c\otimes M \ar@{ >->}[d] \\
\sfB\mcC_c\otimes_\Delta\widetilde N \ar[r] & \mcC_c\otimes N
}\]
whose pushout corner map also lies in $\sfW_0$. In addition, it is a map between cofibrant objects of the under category $\mcC_c \otimes M \backslash \mcM$ and is thus preserved by pushouts along maps with domain $\mcC_c \otimes M$.

Let us study the augmentation $\sfB\mcC_\bullet \otimes_{\Delta \otimes \mcC} \widetilde X \to X$ in the case that $X$ is cofibrant. Let $0\ra X$ be given as a transfinite composite of a chain $X_n$ in which each map $X_{n-1}\ra X_n$ is a pushout of a generating cofibration $\mcC_c\otimes M\ra\mcC_c\otimes N$; by Proposition~\ref{p:generating_strong_cofibrations}, we may assume $M=X_{n-1}c$. We obtain a compatible frame on $M$ from that on $X_{n-1}$ by setting $\widetilde M=\widetilde X_{n-1}c$. Next we construct a compatible frame for $N$ as above. It is easy to see that the pushout in the diagram
\[\xymatrix{
\mcC_c\otimes\widetilde M \ar[r] \ar@{ >->}[d] & \widetilde X_{n-1} \ar@{ >->}[d] \\
\mcC_c\otimes\widetilde N \ar[r] & \widetilde X_n \po
}\]
is a frame for $X_n$. Then for the right face of
\[\xymatrix@!=1.5pc{
& \mcC_c\otimes M \ar[rr] \ar@{ >->}[dd]|!{[dl];[dr]}\hole & & X_{n-1} \ar@{ >->}[dd] \\
\sfB\mcC_c \otimes_\Delta \widetilde M \ar[rr] \ar@{ >->}[dd] \ar[ur] & & \sfB\mcC_\bullet \otimes_{\Delta \otimes \mcC} \widetilde X_{n-1} \ar@{ >->}[dd] \ar[ur] & \\
& \mcC_c\otimes N \ar[rr]|!{[ur];[dr]}\hole & & X_n \poxy & \\
\sfB\mcC_c \otimes_\Delta \widetilde N \ar[rr] \ar[ur] & & \sfB\mcC_\bullet \otimes_{\Delta \otimes \mcC} \widetilde X_n \ar[ur] \poxy
}\]
the map from the pushout to $X_n$ also belongs to $\sfW_0$ (as a pushout of a map from $\sfW_0$). Therefore, if inductively $\sfB\mcC_\bullet \otimes_{\Delta \otimes \mcC} \widetilde X_{n-1}\ra X_{n-1}$ belongs to $\sfW_0$, so does $\sfB\mcC_\bullet \otimes_{\Delta \otimes \mcC} \widetilde X_n\ra X_n$. Taking colimit over $n$, we see that $\sfB\mcC_\bullet \otimes_{\Delta \otimes \mcC} \widetilde X\ra X$ also belongs to $\sfW_0$.

For each cofibrant diagram $X$, we have constructed a frame $\widetilde X$ such that the augmentation $\sfB\mcC_\bullet \otimes_{\Delta \otimes \mcC} \widetilde X\ra X$ belongs to $\sfW_0$. Next we prove that, in fact, any weak equivalence between cofibrant objects belongs to $\sfW_0$. Therefore, let $f \colon X\ra Y$ be a weak equivalence between cofibrant objects and let $\widetilde X$ and $\widetilde Y$ be frames for which the augmentations belong to $\sfW_0$ as above. Construct the indicated pushout in the square
\[\xymatrix{
\Delta_0\cdot X \ar@{ >->}[r] \ar[d]^\sim & \widetilde X \ar[rr]^-\sim \ar[dd] \ar@/^10pt/[ddr]^\sim & & \const_\Delta X \ar[d]^\sim \\
\Delta_0\cdot Y \ar@{ >->}[d] & & & \const_\Delta Y \\
\widetilde Y \ar@{ >->}[r] \ar@/_20pt/[rr]_\sim & \bullet \po \ar@{ >->}[r] & \widetilde Z \ar[ur]_\sim
}\]
and factor the canonical map from this pushout to $\const_\Delta Y$ into a Reedy cofibration followed by a weak equivalence. The middle object $\widetilde Z$ then becomes another frame on $Y$. It is connected to $\widetilde X$ and $\widetilde Y$ by weak equivalences. Thus, upon applying $\sfB\mcC_\bullet \otimes_{\Delta \otimes \mcC} -$, we obtain the following diagram
\[\xymatrix{
\sfB\mcC_\bullet \otimes_{\Delta \otimes \mcC} \widetilde X \ar[r]^-\sim \ar[d]_\sim & \sfB\mcC_\bullet \otimes_{\Delta \otimes \mcC} \widetilde Z \ar[d] & \sfB\mcC_\bullet \otimes_{\Delta \otimes \mcC} \widetilde Y \ar[l]_-\sim \ar[d]^\sim \\
X \ar[r]_f & Y & Y \ar@{=}[l]
}\]
with the indicated maps lying in $\sfW_0$. This finishes the proof of the theorem.
\qed

\section{Simplicial and cosimplicial diagrams}

In this section, the simplicial and cosimplicial diagrams are understood to be equipped with the Reedy structures.

\begin{proposition}\label{p:BC_Reedy_cofibrant}
Let $\mcC$ be a small locally flat $\mcV$-category. Then $\sfB \mcC_\bullet$ is Reedy cofibrant as an object of $\simp[\mcC^\op \otimes \mcC,\mcV]$, where $\mcV$ is equipped with flat maps.
\end{proposition}

\begin{proof}
We need to compute the latching objects of $\sfB \mcC_\bullet$ and their maps into $\sfB \mcC_\bullet$. These are clearly certain coproducts
\[L_n \sfB \mcC_\bullet = \sum_{c_0,\ldots,c_n} \mcC_\deg(-,c_0,\ldots,c_n,-),\]
which we will now describe. Consider the 1-cubes in $\mcV$
\[X_i = \begin{cases}
0\ra\mcC(c_{i-1},c_i) & \textrm{ if }c_{i-1}\neq c_i \\
\bbone\ra\mcC(c_{i-1},c_i) & \textrm{ if }c_{i-1}=c_i.
\end{cases}\]
They clearly depend on $c_0,\ldots,c_n$ but we will not reflect this in the notation. The map in question $L_n \sfB \mcC_\bullet \to \sfB_n \mcC_\bullet$ is then isomorphic to
\[\sum_{c_0,\ldots,c_n}\mcC(c_n,-) \otimes \pcm(X_n \otimes \cdots \otimes X_1) \otimes \mcC(-,c_0).\]
Since each $X_i$ is flat in $\mcV$ and $-\otimes-\col\mcV\otimes\mcV\to\mcV$ is left closed with respect to flat maps, the claim is proved by $\mcC(c_n,-) \otimes \mcC(-,c_0) \cong (\mcC^\op \otimes \mcC)_{(c_0,c_n)}$.
\end{proof}

One may view an \emph{augmented simplicial object} as a map $\varepsilon \colon X\ra\Delta^0\cdot X_{-1}$ of simplicial objects. An \emph{extra degeneracy} $s_{-1}$ then provides a section of this map, namely, in degree $n$, the component $X_{-1}\ra X_n$ is the composition of the extra degeneracy maps. It is convenient to organize an augmented simplicial object with extra degeneracy into a diagram $\widehat X \colon \Delta_\bot^\op\ra\mcM$. The category $\Delta_\bot$ is isomorphic to a subcategory of $\Delta$ of those monotone maps which preserve the bottom element. We will however think of the objects of $\Delta_\bot$ as objects of $\Delta$ with a bottom element adjoined and called $-1$. Thus $[n]_+$ will stand for the ordered set $\{-1,0,\ldots,n\}$ and $X_n$ for $X[n]_+$. We observe that $\Delta_\bot$ has an initial and terminal object $[-1]_+$. The unique map $[-1]_+\ra[n]_+$ yields exactly the augmentation $X_n\ra X_{-1}$ while the extra degeneracy $X_{-1}\ra X_n$ is the image of the unique map $[n]_+\ra[-1]_+$.

\begin{proposition}\label{p:BC_c_Reedy_cofibrant}
Let $\mcC$ be a small locally flat $\mcV$-category. Then $\sfB \mcC_c$ is Reedy cofibrant as an object of $\simp_\bot[\mcC,\mcV]$, where $\mcV$ is equipped with flat maps.
\end{proposition}

\begin{proof}
The proof is similar to the previous one but this time $L_{n_+} \sfB \mcC_\bullet \to \sfB_{n_+} \mcC_\bullet$ is isomorphic to
\[\sum_{c_0,\ldots,c_n}\mcC(c_n,-) \otimes \pcm(X_n \otimes \cdots \otimes X_1 \otimes X_0).\qedhere\]
\end{proof}

A \emph{frame} on an object $W^{-1} \in \mcQ$ is a factorization in $\cosimp\mcQ$
\begin{equation}\label{eq:frame}
\Delta_0 \cdot W^{-1} \rightarrowtail W \simarrow \const_{\Delta} W^{-1}
\end{equation}
of the canonical map into a Reedy cofibration followed by a weak equivalence.

\begin{proposition}\label{p:extra_degenerate_realization}
Let $W$ be a frame on a cofibrant object $W^{-1}$ and $X$ a simplicial object augmented by $X_{-1}$ and equipped with extra degeneracy in such a way that the corresponding diagram $\widehat X\in\simp_\bot\mcP$ (with the augmentation and the extra degeneracies added in) is Reedy cofibrant. Then the map
\[X\otimes_\Delta W \lra X_{-1}\otimes W^{-1},\]
obtained by tensoring the augmentation $\varepsilon $of $X$ with the weak equivalence in \eqref{eq:frame}, admits a section by a trivial cofibration.
\end{proposition}

\begin{proof}
There is a canonical functor (the adjoining of the bottom element $-1$)
\[\iota=(-)_+ \colon \Delta\ra\Delta_\bot.\]
The simplicial object $X$ is obtained as $X=\iota^*\widehat X$. Let $\widehat W\in\cosimp_\bot\mcQ$ denote the left Kan extension $\widehat W=\iota_!W$. Explicitly,
\[\widehat W^n=\int\nolimits^{[k]\in\Delta}\Delta_\bot([k]_+,[n]_+)\cdot W^k\cong\int\nolimits^{[k]\in\Delta}\Delta([k],[n+1])\cdot W^k\cong W^{n+1},\]
and it is thus isomorphic to the restriction of $W$ along the embedding $\Delta_\bot\hra\Delta$, showing that $\widehat W$ is homotopy locally constant. Next, we will show that it is Reedy cofibrant.

To this end, we observe that $j^n \otimes_{\Delta_\bot} \iota_!W \cong \iota^*j^n \otimes_\Delta W$ and that $\iota^*j^n$ is isomorphic to the inclusion $\Lambda^{n+1}_0\ra\Delta^{n+1}$ of the $0$-th horn in $\Delta^{n+1}$. Since any monomorphism is generated frmo $\partial\Delta^k\ra\Delta^k$, we see that $\iota^* j^n$ is a Reedy cofibration and Proposition~\ref{p:coend_Reedy} shows that the latching maps for $W$ are indeed cofibrations.

We may now transfer all the coends to $\Delta_\bot$ since
\[X\odot_\Delta W\cong\widehat X\odot_{\Delta_\bot}\widehat W.\]
The map from the statement is induced by a weak equivalence $\widehat W \to \const_{\Delta_\bot} W^{-1}$.

The section is induced by the transformation of ``weights''
\[(\Delta_\bot)_{-1}\cdot W^{-1} \lra (\Delta_\bot)_{-1}\cdot W^0 = \sk^{-1} \widehat W \lra \widehat W\]
which is a composition of two Reedy trivial cofibrations. Since $-\odot_{\Delta_\bot}-$ is left Quillen with respect to the Reedy structures and $\widehat X$ is Reedy cofibrant, the claim follows.
\end{proof}

\begin{lemma}\label{l:coend_restriction_extension}
If $\alpha \colon \mcC\ra\mcD$ is a functor and $F \colon \mcD^\op \otimes \mcC \ra \mcM$ a bifuntor then
\[\smallint\nolimits^\mcC(\alpha\otimes\id)^*F\cong\smallint\nolimits^\mcD(\id\otimes\alpha)_!F\]
Moreover the left Kan extension $(\id\otimes\alpha)_!F$ may be computed as
\[\mcD^\op\xra{F}[\mcC,\mcM]\xra{\alpha_!}[\mcD,\mcM]\]
\end{lemma}

\begin{proof}
First we derive the formula for a special case of a $\mcV$-balanced bifunctor
\begin{align*}
W\otimes_\mcC \alpha^*X & =Wc\otimes_{c\in\mcC}X\alpha c\cong Wc\otimes_{c\in\mcC}(\mcD(d,\alpha c)\otimes_{d\in\mcD}Xd) \\
& \cong (Wc\otimes_{c\in\mcC}\mcD(d,\alpha c))\otimes_{d\in\mcD}Xd\cong(\alpha_!W)d\otimes_{d\in\mcD}Xd \\
& =\alpha_!W\otimes_\mcD X.
\end{align*}
This aplies in particular to the tensor action. The remainder of the proof is the computation
\begin{align*}
\smallint\nolimits^\mcC(\alpha\otimes\id)^*F & \cong\mcC(-,-)\otimes_{\mcC^\op\otimes\mcC}(\alpha\otimes\id)^*F \\
& \cong(\alpha\otimes\id)_!\mcC(-,-)\otimes_{\mcD^\op\otimes\mcC}F \\
& \cong\mcD(-,\alpha-)\otimes_{\mcD^\op\otimes\mcC}F \\
& \cong(\id\otimes\alpha)^*\mcD(-,-)\otimes_{\mcD^\op\otimes\mcC}F \\
& \cong\mcD(-,-)\otimes_{\mcD^\op\otimes\mcD}(\id\otimes\alpha)_!F \\
& \cong\smallint\nolimits^\mcD(\id\otimes\alpha)_!F,
\end{align*}
where $((\id\otimes\alpha)_!F)(d',d) \cong \mcD(\alpha-,d) \otimes_\mcC F(d',-)$.
\end{proof}

\section{Lifting properties of cofibrations and fibrations in diagram categories}

For this part, we assume that $\mcM$ is equipped with a reduced cofibration and a reduced fibration structure sharing the same class of weak equivalences and that trivial cofibrations lift against fibrations and cofibrations against trivial fibrations. We call such a structure a \emph{basic model category} (defined in \cite{Weibel}; the name Thomason model category is reserved when $\mcM$ admits functorial factorizations).

It is then rather obvious that cofibrations lift against pointwise trivial fibrations and pointwise trivial cofibrations against fibrations. From this, it easily follows that the diagram category $[\mcC,\mcM]$ equipped with (projective) cofibrations, (injective) fibrations and (pointwise) weak equivalences is a basic model category.

Our aim now is to prove a stronger version where injective fibrations are weakened to pointwise fibrations.

\begin{theorem}
The diagram category $[\mcC,\mcM]$ equipped with projective cofibrations, pointwise fibrations and pointwise weak equivalences is a basic model category.
\end{theorem}

\begin{proof}
Obviously, it remains to show that trivial cofibrations lift against pointwise fibrations:
\[\xymatrix{
	A \ar[r]^-f \ar@{ >->}[d]_-i^-\sim & X \ar[d]^-p \\
	B \ar[r]_-g \ar@{-->}[ru] & Y
}\]
We first replace $p$ by an injective fibration $\widehat p \colon \widehat X \to \widehat Y$ so that a diagonal $B \to \widehat X$ exists. Next, using the homotopy right lifting property of Proposition~\ref{p:homotopy_right_lifting}, we find a map $d \colon B \to X$ together with a right homotopy of maps $B \to \widehat X$ in the triangle $BX\widetilde X$ that is constant when restricted to $A$.
\[\xymatrix{
	A \ar[r]^-f \ar@{ >->}[d]_-i^-\sim & X \ar[r]^-\sim \ar[d]_-p \POS[];[rd]**{}?<>(.15)*{\scriptstyle\mathrm{h}} & \widehat X \ar@{->>}[d]^-{\widehat p} \\
	B \ar[r]_-g \ar@{-->}[ru]^-d \ar@{-->}@/_5pt/[rru] & Y \ar[r]_-\sim & \widehat Y
}\]
Next, using Proposition~\ref{p:left_right_homotopy}, we replace this right homotopy by a left homotopy with respect to a ``good'' cylinder of Lemma~\ref{l:good_frame}. Composing with $\widehat p$, we get a left homotopy of maps $B \to \widehat Y$ that gives the bottom map in the following square. The top map consists of a constant homotopy $IA \to Y$ and the map $\partial IB \to Y$ composed of a pair of maps $pd,\,g \colon B \to Y$.
\[\xymatrix{
	\partial IB +_{\partial IB} IA \ar[r] \ar@{ >->}[d] & Y \ar[d]^-\sim \\
	IB \ar[r] \ar@{-->}[ru] & \widehat Y
}\]
Thus, according to Proposition~\ref{p:homotopy_right_lifting} again, a diagonal exists for which the upper triangle commutes strictly and the lower up to a left homotopy. Therefore, this diagonal is a left homotopy between $pd$ and $g$ that is constant when restricted to $A$. Finally, an application of Proposition~\ref{p:strictification_of_homotopy_diagonals} produces a homotopy of $d$ to a strict diagonal in the original square.
\end{proof}

\begin{proposition}\label{p:homotopy_right_lifting}
A map $f \colon X\ra Y$ between fibrant objects is a weak equivalence if and only if it has a homotopy right lifting property with respect to cofibrations.
\end{proposition}

\begin{proof}
Let $\widetilde X\ra X$ be a cofibrant replacement and factor the composition $\widetilde X\ra X\ra Y$ into a cofibration and a weak equivalence as in
\[\xymatrix{
\widetilde X \ar[r]^-\sim \ar@{ >->}[d] & X \ar[d]^f \\
\widetilde Y \ar[r]_-\sim \ar@{-->}[ur] & Y
}\]
By the assumed HRLP, a diagonal exists for which both triangles commute in the homotopy category. Thus, $f$ becomes an isomorphism in the homotopy category and as such has to be a weak equivalence.

In the opposite direction, one uses a path object $Y\ra PY\ra Y\times Y$ to produce a trivial fibration
\[\xymatrix{
A \ar[r] \ar@{ >->}[d] & X \ar[r]^-\sim & X\times_YPY \ar@{->>}[d]^-\sim \ar[r] \POS[];[dr]**{}?(.35)*{\scriptstyle\mathrm h} & X \ar[dl] \\
B \ar[rr] \ar@{-->}[urr] & & Y &
}\]
and the indicated lift exists by the usual lifting properties. The triangle on the right commutes up to a right homotopy.
\end{proof}

\begin{proposition}[homotopy extension property]\label{p:strictification_of_homotopy_diagonals}
Assume that in the diagram
\[\xymatrix{
A \ar[r] \ar@{ >->}[d] & X \ar@{->>}[d] \\
B \ar[r] \ar@{-->}[ur] & Y
}\]
there exists a diagonal for which the lower triangle commutes strictly and the upper triangle commutes up to a fibrewise homotopy (a homotopy constant when projected to $Y$). Then there exists a strict diagonal filler.
\end{proposition}

\begin{proof}
This is easy
\[\xymatrixb{
A \ar[r] \ar@{ >->}[d] & IA\sqcup_AB \ar[r] \ar@{ >->}[d]_\sim & X \ar@{->>}[d] \\
B \ar[r] & IB \ar[r] \ar@{-->}[ur] & Y
}{
B \ar[r] & IB \ar[r] & Y
}\qedhere\]
\end{proof}

\begin{proposition}\label{p:left_right_homotopy}
Let $A \rightarrowtail B$ be a strong cofibration and $X$ a fibrant object. Then two maps $f_0,\,f_1 \colon B\ra X$ are left homotopic relative to $A$ iff they are right homotopic relative to $A$.
\end{proposition}

\begin{proof}
Let $F \colon B \ra PX$ be a right homotopy whose restriction to $A$ is a constant homotopy at $g \colon A \to X$. We consider
\[\xymatrix{
\partial IB +_{\partial IA} IA \ar[rr]^-{[[F,if_1],igp]} \ar@{ >->}[d] & & PX \ar@{->>}[d]^-{p_1}_\sim \ar[r]^-{p_0} & X \\
IB \ar[r]_p \ar@{-->}[urr]^H & B \ar[r]_-{f_1} & X
}\]
The required left homotopy is the composition $p_0H$ where $H$ is any diagonal in the above square.
\end{proof}

\begin{lemma}\label{l:good_frame}
For every strong cofibration $A \ra B$ in $[\mcC,\mcM]$, there exist cylinders $IA$ and $IB$ such that the pushout corner map in
\[\xymatrix{
A \ar[r]^-{i_0} \ar@{ >->}[d] & IA \ar@{ >->}[d] \\
B \ar[r]_-{i_0} & IB
}\]
is a good trivial cofibration.
\end{lemma}

\begin{proof}
The cylinders are constructed from frames as $IA = \widetilde A^1$. Let $B$ be a transfinite composite of a chain $B_n$ with $B_{n-1} \to B_n$ obtained as a pushout of $\mcC_c \otimes M \to \mcC_c \otimes N$. Then, using the cylinders constructed in the course of the proof of Theorem~\ref{t:weak_equivalences_generation}, we have the following diagram
\[\xymatrix@!=1pc{
& \mcC_c\otimes \widetilde M^1 \ar[rr] \ar[dd]|!{[dl];[dr]}\hole & & (\widetilde B_{n-1})^1 \ar[dd] \\
\mcC_c\otimes \widetilde M^0 \ar[rr] \ar[dd] \ar[ur] & & (\widetilde B_{n-1})^0 \ar[dd] \ar[ur] & \\
& \mcC_c\otimes \widetilde N^1 \ar[rr]|!{[ur];[dr]}\hole & & (\widetilde B_n)^1 \poxy & \\
\mcC_c\otimes \widetilde N^0 \ar[rr] \ar[ur] & & (\widetilde B_n)^0 \ar[ur] \poxy
}\]
Since the pushout corner map of the left face is a good trivial cofibration, the same is true for the right face (as it is a horizontal pushout of the left face). This completes the inductive step.
\end{proof}

\vskip 20pt
\vfill
\vbox{\footnotesize%
\noindent\begin{minipage}[t]{0.5\textwidth}
{\scshape
Luk\'a\v{s} Vok\v{r}\'inek}
\vskip 2pt
Department of Mathematics and Statistics,\\
Masaryk University,\\
Kotl\'a\v{r}sk\'a~2, 611~37~Brno,\\
Czech Republic
\vskip 2pt
\url{koren@math.muni.cz}
\end{minipage}
}

\end{document}